\documentclass{article} 
\usepackage[utf8]{inputenc}
\usepackage[margin =2.4cm]{geometry}
\usepackage{amssymb}
\usepackage{amsmath}
\usepackage{amsthm}
\usepackage{enumerate} 
\usepackage{graphicx}
\usepackage{csquotes}
\usepackage{subcaption}
\usepackage{authblk}
\usepackage{hyperref}

\newtheorem{theorem}{Theorem}

\newtheorem{lemma}[theorem]{Lemma}
\newtheorem{definition}{Definition}[section]
\newtheorem{example}{Example}

\title{\LARGE{Insights into Weighted Sum Sampling Approaches for  \protect\\ Multi-Criteria Decision Making Problems}}

\author{
  \begin{minipage}[t]{0.45\textwidth}
    \centering
    {\Large Aled Williams}\\
    Department of Mathematics\\
    London School of Economics and Political Science\\
    London, UK\\
    \texttt{a.e.williams1@lse.ac.uk}
  \end{minipage}%
  \hfill
  \begin{minipage}[t]{0.45\textwidth}
    \centering
    {\Large Yilun Cai}\\
    Department of Statistics\\
    University of Chicago\\
    Chicago \\ Illinois, USA\\
    \texttt{yiluncai@uchicago.edu}
  \end{minipage}
}

\begin{document}
	\date{}
	\maketitle

\begin{abstract}
In this paper we explore several approaches for sampling weight vectors in the context of weighted sum scalarisation approaches for solving multi-criteria decision making (MCDM) problems. This established method converts a multi-objective problem into a (single) scalar optimisation problem. It does so by assigning weights to each objective. We outline various methods to select these weights, with a focus on ensuring computational efficiency and avoiding redundancy. The challenges and computational complexity of these approaches are explored and numerical examples are provided. The theoretical results demonstrate the trade-offs between systematic and randomised weight generation techniques, highlighting their performance for different problem settings. These sampling approaches will be tested and compared computationally in an upcoming paper. 

\vspace{1.0mm}
\noindent \textbf{Keywords}: multi-criteria decision making problems, weighted sum method, weight sampling, Pareto efficiency,  Pareto frontier, nondominated points.
\end{abstract}


\section{Introduction}
\subsection{Multi-Criteria Decision Making Problems}
We consider multi-criteria optimisation problems of form
\begin{equation} \label{multi_problem}
\begin{aligned}
\text{\enquote{minimise}} \quad \big( f_1(\boldsymbol{x}), f_2(\boldsymbol{x}), &\ldots, f_p(\boldsymbol{x}) \big) \\
\text{subject to } \quad \boldsymbol{x} &\in \mathcal{X},
\end{aligned}
\end{equation}
where $\mathcal{X} \subset \mathbb{R}^n$ denotes the \textit{feasible set} (or set of alternatives of the decision problem) and $f_i : \mathbb{R}^n \rightarrow \mathbb{R}$ for each $i \in \{1,2,\ldots, p\}$. Let us assume, for simplicity of presentation, that each objective $f_i$ is linear. It should be noted that each objective function $f_i(\boldsymbol{x})$ represents a different criterion (or aspect) of the decision making problem. The aim is to minimise the $p$ objective functions simultaneously, which typically involves a trade-off between objectives. 

\subsection{Efficiency and Nondominance}
To clarify \enquote{minimise} in \eqref{multi_problem}, we formally define efficient solutions and nondominated points, which are defined by the component-wise order over the $p$ objectives. Let $\mathcal{X}$ denote the feasible set of solutions to the above problem. Further, denote by $\mathcal{Y}:=c(\mathcal{X})$ the objective function mapping of the feasible set $\mathcal{X}$, where $c = (\boldsymbol{c}_1, \boldsymbol{c}_2, \ldots, \boldsymbol{c}_p)$ for $\boldsymbol{c}^T_i \in \mathbb{R}^n$ for each $i$. Note that $\mathcal{X} \subset \mathbb{R}^n$ and $\mathcal{Y} \subset \mathbb{R}^p$. 

\begin{definition}
A feasible solution $\boldsymbol{x}^* \in \mathcal{X}$ is called efficient (or Pareto optimal) if there is no other $\boldsymbol{x} \in \mathcal{X}$ such that $c(\boldsymbol{x}) \le c(\boldsymbol{x}^*)$, i.e. no other feasible $\boldsymbol{x}$ satisfies $\boldsymbol{c}_i^T \boldsymbol{x} \le \boldsymbol{c}_i^T \boldsymbol{x}^*$ for all $i \in \{1,2,\ldots, p\}$ and $\boldsymbol{c}_j^T \boldsymbol{x} < \boldsymbol{c}_j^T \boldsymbol{x}^*$ for at least one $j \in \{1,2,\ldots, p\}$. If $\boldsymbol{x}^*$ is efficient, then $c(\boldsymbol{x}^*)$ is called a nondominated point. 
\end{definition}

In other words, a solution $\boldsymbol{x}^*$ is efficient if there is no $\boldsymbol{x} \in \mathcal{X}$ such that 
$$
\boldsymbol{c}_k^T \boldsymbol{x} \le \boldsymbol{c}_k^T \boldsymbol{x}^* \text{ for } k=1,2,\ldots,p
$$ 
and 
$$
\boldsymbol{c}_l^T \boldsymbol{x} < \boldsymbol{c}_l^T \boldsymbol{x}^* \text { for some } l \in \{1,2,\ldots,p\}.
$$ 
Informally, an efficient solution is a solution that cannot be improved in any of the objectives without degrading at least one of the other objectives. Thus, the fundamental importance of efficiency lies in the fact that any solution that is not efficient cannot represent the most preferred alternative for a decision maker. Next, let us define weakly efficient solutions and nondominated points.

\begin{definition}
A feasible solution $\boldsymbol{x}^* \in \mathcal{X}$ is called weakly efficient (or weakly Pareto optimal) if there is no other $\boldsymbol{x} \in \mathcal{X}$ such that $c(\boldsymbol{x}) < c(\boldsymbol{x}^*)$, i.e. no feasible $\boldsymbol{x}$ satisfies $\boldsymbol{c}_i^T \boldsymbol{x} < \boldsymbol{c}_i^T \boldsymbol{x}^*$ for all $i \in \{1,2,\ldots, p\}$. If $\boldsymbol{x}^*$ is weakly efficient, then $c(\boldsymbol{x}^*)$ is called weakly nondominated.
\end{definition}

It follows from the above definitions that 
$$
\mathcal{Y}_N \subset \mathcal{Y}_{w N} \subset \mathcal{Y} \subset \mathbb{R}^p 
$$
and 
$$
\mathcal{X}_E \subset \mathcal{X}_{w E} \subset \mathcal{X} \subset \mathbb{R}^n,$$
where $\mathcal{Y}_N$, $\mathcal{Y}_{w N}$, $\mathcal{X}_E$ and $\mathcal{X}_{wE}$ denote the set of all nondominated points, weakly nondominated points, efficient solutions and weakly efficient solutions, respectively. Informally, a weakly efficient solution is a solution for which there is no way to improve every objective simultaneously while remaining feasible. Note that the images $\mathcal{Y}_N$ and $\mathcal{Y}_{wN}$ are often called the Pareto frontier (or the Pareto front or nondominated front) and the weak Pareto frontier, respectively. 


\subsection{An Introduction to Weighted Sum Scalarisation}
The traditional approach to solving problems with multi-criteria such as \eqref{multi_problem} is by scalarisation, which involves formulating a single objective optimisation problem that is related to the multi-criteria problem. We begin by outlining one of the most commonly applied scalarisation techniques, namely the weighted sum scalarisation approach, before discussing more formal details around weight selection later. To introduce the method, let us once more denote by $\mathcal{X}$ the feasible set of solutions to problem \eqref{multi_problem}. 

\vspace{2.0mm}

It is important to emphasise that our study is motivated by the growing need for explainable, non-black-box methods in decision-making, in alignment with legislative demands and stakeholder transparency expectations within operational research applications.

\vspace{2.0mm}

The weighted sum method (WSM) converts the original problem to 
\begin{equation} \label{WSM_Definition}
\min _{\boldsymbol{x} \in \mathcal{X}} \, \sum_{i=1}^p \lambda_i \, \boldsymbol{c}_i^T\boldsymbol{x} = \lambda_1 \boldsymbol{c}_1^T \boldsymbol{x} + \cdots + \lambda_p \boldsymbol{c}_p^T \boldsymbol{x}, 
\end{equation}
where $\sum_{k=1}^p \lambda_i=1$ and $\lambda_i \ge 0$ for all $i \in \{1,2,\ldots,p\}$. Note that this approach converts the $p$ objectives into an aggregated scalar objective function by assigning each objective function a weighting factor, before summing yields the overall (single) objective function. Each (original) objective is given a weight to denote its relative importance during the overall aggregation. The method enables the computation of weakly efficient solutions by successively varying the weights $\lambda_i$ for convex problems. Being a little more precise, the following result connecting convexity with efficient solutions is known (see e.g. \cite[Theorem 3.1.4]{miettinen1999nonlinear}). 

\begin{theorem}
Suppose the multi-criteria optimisation problem \eqref{multi_problem} is convex. If $\boldsymbol{x}^*$ is efficient (or Pareto optimal), then there exists a weighting vector 
$\lambda = (\lambda_1, \lambda_2, \ldots, \lambda_p)$ with $\lambda_i \ge 0$ for each $i \in \{1,2,\ldots,p\}$ and $\sum_{i=1}^p \lambda_i = 1$ such that $\boldsymbol{x}^*$ is a solution to the problem \eqref{WSM_Definition}.
\end{theorem}

Thus, the above result suggests that any Pareto optimal solution of a convex multi-criteria optimisation problem can be found via the weighted sum method. The method may however work poorly for non-convex problems, such as a multi-criteria set covering or travelling salesman problems. It should be noted that different approaches for varying the weights are outlined in the subsequent section. 

\vspace{2.0mm}

It should be noted for completeness that if the underlying problem is non-convex, then other scalarisation techniques are (perhaps) more appropriate for finding weakly efficient solutions. The celebrated $\varepsilon$-constraint method \cite{Haimes1971}, hybrid method \cite{Guddat1985}, Benson's Method \cite{benson1978existence}, or the elastic constraint method (see e.g. \cite{Ehrgott2002, tanino2003method, holder2003designing}) are examples of such scalarisation approaches. 


\section{Literature Review}
The weighted sum method (WSM) and its variants are widely used in decision-making and optimisation due to their simplicity and flexibility in evaluating multiple criteria. Applications span diverse fields such as technology, energy, urban planning, and multi-objective optimisation, showcasing its broad relevance across complex problems. The purpose of the following literature review is to summarise existing research on the use of the WSM in multi-criteria decision making (MCDM) problems.

\vspace{2.0mm}

The study, \cite{sarika2012server}, applied WSM and a revised decision model to rank servers from IBM, HP, and Sun Microsystems, considering both objective and subjective criteria to identify optimal business solutions. A review of MCDM approaches for evaluating energy storage systems, presented in \cite{baumann2019review}, considered economic, technical, and environmental factors. Research in \cite{yao2023advanced} focused on industrial informatics applications in road engineering, particularly AI-based systems designed to enhance the construction, maintenance, and safety of road infrastructure.

\vspace{2.0mm}

In \cite{lee2018comparative}, multiple MCDM methods, including WSM, were compared for ranking renewable energy sources in Taiwan, with hydropower identified as the most suitable option. An analysis presented in \cite{stanujkic2013comparative} demonstrated how various MCDM methods produce differing rankings, particularly in banking contexts. The importance of developing standardised approaches for land-use optimisation and integrating sustainability in urban planning was emphasised in \cite{rahman2021multi}.

\vspace{2.0mm}

The literature review in \cite{mofidi2020intelligent} explored intelligent building operations, highlighting the necessity of balancing energy efficiency and occupant comfort. Adaptive Weighted Sum (AWS) methods for multi-objective optimisation, aimed at improved exploration of Pareto fronts, were developed and successfully applied in \cite{kim2006adaptive}. Recent novel weighting methods for MCDM, including CILOS and MEREC, were reviewed in depth in \cite{ayan2023comprehensive}.

\vspace{2.0mm}

New decision-making methodologies for ranking non-dominated points within multi-objective optimisation problems (MOPs) were introduced in \cite{dolatnezhadsomarin2024new}. Additionally, the cascaded weighted sum method (CWS) proposed in \cite{jakob2014pareto} provided an alternative approach to traditional Pareto optimisation, particularly suitable for scheduling problems. A weighted sum-based method to enhance the detection of rare genetic variants in association tests for genetically heterogeneous diseases was presented in \cite{madsen2009groupwise}.

\vspace{2.0mm}

A review of Pareto and scalarisation techniques in multi-objective optimisation (MOO) highlighted their practicality and ease of implementation in real-world applications \cite{gunantara2018review}. In the context of financial decision-making, \cite{Grodzevich2006} demonstrated the value of normalisation techniques and the effectiveness of weighted sum approaches in portfolio optimisation. The Adaptive Weighted Sum (AWS) method was further refined in \cite{kim2005adaptive} to better address bi-objective optimisation problems.

\vspace{2.0mm}

A control-function-based method for multi-objective optimisation, introduced in \cite{augusto2012new}, reduced computational requirements by avoiding the explicit construction of the Pareto set while ensuring solution optimality. The Evolutionary Dynamic Weighted Aggregation (EDWA) approach in \cite{jin2001dynamic} improved traditional weighting schemes by dynamically adjusting weights to better capture both convex and concave regions of the Pareto front. The entropy weights method, reviewed in \cite{kumar2021revealing}, demonstrated effectiveness in machining operations and showed potential for broader application due to its objective weighting capabilities.

\vspace{2.0mm}

A comprehensive survey of evolutionary algorithms for MOPs with irregular Pareto fronts was presented in \cite{hua2021survey}, which categorised algorithmic strategies and outlined the challenges involved in solving complex, real-world problems. A benchmark suite of 16 bound-constrained multi-objective problems, including mixed-integer variants, was introduced in \cite{tanabe2020easy} to support performance evaluation of optimisation algorithms. Dynamic weight adjustment mechanisms to improve decomposition-based methods for irregular Pareto fronts were proposed in \cite{li2020weights}, enabling adaptive control of weight distribution and archive maintenance during optimisation.

\vspace{2.0mm}

A multi-objective genetic algorithm for flowshop scheduling was presented in \cite{murata1996multi}, incorporating variable weighting and elite preservation to improve diversity and capture concave Pareto regions. The clustering-ranking evolutionary algorithm (crEA) developed in \cite{cai2015clustering} showed effective performance across many-objective benchmark problems by enhancing both convergence and diversity. An extension of the expected hypervolume improvement (EHVI) criterion was proposed in \cite{feliot2019user}, using weighted preferences and sequential Monte Carlo (SMC) techniques to guide Bayesian optimisation in expensive black-box settings.

\vspace{2.0mm}

The W-HypE algorithm introduced in \cite{brockhoff2013directed} used weighted hypervolume indicators and Monte Carlo sampling to steer optimisation towards user-preferred regions in high-dimensional objective spaces. A comparative study in \cite{hughes2005evolutionary} evaluated NSGA-II, Multiple Single Objective Pareto Sampling (MSOPS), and repeated single objective optimisations (RSO), concluding that MSOPS and RSO outperform NSGA-II in many-objective settings due to their avoidance of Pareto dominance ranking. A preference-based selection method using the minimum Manhattan distance, avoiding subjective weights, was proposed in \cite{chiu2016minimum} as a computationally efficient approach for MCDM problems.

\vspace{2.0mm}

A comprehensive survey of MOPSO applications across various fields was conducted in \cite{lalwani2013comprehensive}, highlighting the algorithm’s flexibility, variants, and practical advantages for complex multi-objective problems. The Dynamical Multi-Objective Evolutionary Algorithm (DMOEA) and the L-optimality concept introduced in \cite{zou2008new} showed improved solution diversity and convergence, especially in many-objective contexts. A two-stage evolutionary algorithm (MaOEA-IT) was proposed in \cite{sun2018new}, independently addressing convergence and diversity, and demonstrating performance improvements over six benchmarked MaOEAs.

\vspace{2.0mm}

A multi-objective bat algorithm (MOBA), developed in \cite{yang2011bat}, incorporated adaptive parameter tuning and weighted sum strategies to address constrained optimisation problems, achieving competitive results on benchmark and real-world design tasks. Enhanced algorithms for sum estimation under both proportional and hybrid sampling settings were presented in \cite{beretta2024better}, improving upon earlier methods such as those in \cite{motwani2007estimating}. These methods introduced tighter complexity bounds—$\mathcal{O}(\sqrt{n}/\epsilon)$ for proportional settings and $\mathcal{O}(n^{1/3}/\epsilon^{4/3})$ for hybrid settings—along with strategies to handle unknown universe sizes. A weight-agnostic constrained sampling technique called WAPS, introduced in \cite{10.1007/978-3-030-17462-0_4}, utilised d-DNNF compilation to achieve significant runtime and scalability improvements over prior approaches such as WeightGen.

\vspace{2.0mm}

An analysis of Markov Chain Monte Carlo (MCMC) sampling for the Winnow multiplicative weight update algorithm was presented in \cite{tao2003analysis}, demonstrating how computational complexity can be reduced while maintaining accuracy, with techniques such as parallel tempering improving sampling efficiency. Scalable parallel algorithms for weighted random sampling were proposed in \cite{hubschle2022parallel}, achieving near-linear speedups across shared- and distributed-memory architectures through communication-efficient data structures. A multi-objective Artificial Bee Colony (ABC) algorithm was applied in \cite{naidu2014multiobjective} to tune PID controllers for load frequency control, outperforming conventional approaches based on key performance indices.

\vspace{2.0mm}

A hybrid optimisation algorithm combining genetic algorithms, grey wolf optimiser, water cycle algorithm, and population-based incremental learning using a weighted sum approach (E-GGWP-W) was developed in \cite{wansasueb2021ensemble} and applied to composite wing design, yielding superior performance on benchmark problems. An MCDM method integrating the Weighted Sum approach with the Step-Wise Weight Assessment Ratio Analysis (SWARA) method was proposed in \cite{Stanujkic2017new}, enabling flexible, consensus-driven decision-making in personnel selection tasks. A weighted sum model for wind turbine selection was employed in \cite{Rehman2017multi}, facilitating the identification of optimal alternatives across 18 commercial turbine options based on five key technical criteria.

\vspace{2.0mm}

A hybrid decision-making approach known as the Weighted Aggregates Sum Product Assessment (WASPAS) method was introduced in \cite{Zavadskas2012optimization}, combining WSM and the Weighted Product Model (WPM) to enhance ranking precision in MCDM problems. The Weighted Sum Preferred Levels of Performances (WS PLP) method presented in \cite{Karabasevic2018weighted} integrated SWARA weighting with preferred performance levels, offering a practical approach for personnel selection in human resource management. A hybrid pathfinding algorithm called Weighted Sum-Dijkstra’s Algorithm (WSDA), combining WSM and Dijkstra’s algorithm, was proposed in \cite{Hua2018weighted}, enabling efficient multi-criteria path selection based on normalised and weighted attributes such as cost, distance, and travel time.

\vspace{2.0mm}

The study in \cite{sianturi2019implementation} applied the WSM framework to develop a decision support system for selecting football athletes, using weighted evaluation criteria to generate data-driven rankings. In \cite{Harahap2023multi}, the Analytic Hierarchy Process (AHP) was used to derive weights for evaluating university housing options, demonstrating how WSM can be integrated with student preference data to support structured housing decisions. An improved WSM for evaluating weapon systems was introduced in \cite{zhi2013weighted}, which combined subjective expert judgment with objective weighting via grey theory to enhance accuracy and reduce bias in the evaluation process.

\vspace{2.0mm}

A low-complexity algorithm for weighted sum-rate maximisation in MIMO broadcast channels was proposed in \cite{christensen2008weighted}, leveraging alternating optimisation of transmit and receive filters to ensure fast convergence and strong system performance. The application of WSM to medical data for ranking breast carcinoma types was demonstrated in \cite{yong2022analysis}, where ten clinical criteria were used to support objective, preference-based diagnosis. In \cite{fisal2022adaptive}, an adaptive weighted sum bi-objective Bat algorithm (AW-ABBA) was developed for regression testing, tackling the test suite reduction problem by optimising execution time and fault detection simultaneously, and outperforming traditional approaches.

\vspace{2.0mm}

Mukhametzyanov \cite{mukhametzyanov2021specific} compared objective weighting methods in MCDM and proposed improved entropy-based approaches (EWM.df, EWM.dsp, and EWM-Corr) to address limitations in weight assignment and correlated criteria. Keshavarz-Ghorabaee et al. \cite{keshavarz2021determination} introduced MEREC, a novel weighting method that evaluates the impact of removing criteria, showing strong performance and consistency across various matrices. Brauers et al. \cite{Brauers2008multi} applied the MOORA (Multi-Objective optimisation on the basis of Ratio Analysis) method to optimise road construction design, ranking alternatives based on multiple criteria in a real-world case study.

\vspace{2.0mm}

A comparative analysis of objective weighting methods was conducted in \cite{mukhametzyanov2021specific}, which identified limitations in conventional entropy-based models and introduced improved versions (EWM.df, EWM.dsp, and EWM-Corr) to better manage correlated criteria in MCDM settings. The MEREC method, introduced in \cite{keshavarz2021determination}, calculated criterion importance by measuring the performance change caused by criterion removal, offering consistent results across diverse decision matrices. A real-world road construction case study in \cite{Brauers2008multi} demonstrated the application of the MOORA method, which ranked highway design alternatives using a ratio-based multi-objective evaluation approach.

\vspace{2.0mm}

An improved version of the M-PF optimiser, known as the elite Multi-Criteria Decision Making–Pareto Front (eMPF) optimiser, was introduced in \cite{Kesireddy2024elite}. This method integrates multi-objective optimisation with decision-making techniques to efficiently explore and refine Pareto-optimal solutions. A reinforcement learning-based optimisation framework for power generation scheduling was proposed in \cite{Ebrie2024reinforcement}, utilising a multi-agent deep reinforcement learning (MADRL) model to decompose the problem into sequential Markov decision processes. In \cite{Meghwani2020adaptively}, an adaptively weighted decomposition-based evolutionary algorithm (AWMOEA/D) was developed to improve convergence and diversity by modifying scalarisation weights based on crowding distance metrics.

\vspace{2.0mm}

A comprehensive review of the NSGA-II algorithm and its application to multi-objective combinatorial optimisation problems (MOCOPs) was conducted in \cite{verma2021comprehensive}. The review classified research into conventional, modified, and hybrid NSGA-II implementations, and compared their performance with alternative algorithms across a range of problems including assignment, vehicle routing, and knapsack. A fuzzy SWARA-CoCoSo model was proposed in \cite{ulutacs2020location} for selecting optimal logistics centre locations, integrating geographic information systems (GIS) and validating results against other MCDM methods. In \cite{guo2020weighted}, a low-complexity optimisation strategy was developed for reconfigurable intelligent surface (RIS)-assisted multiuser MISO downlink communication systems, maximising the weighted sum-rate (WSR) under varying channel state information conditions.

\vspace{2.0mm}

A hybrid approach combining fuzzy MCDM methods with multi-objective programming was introduced in \cite{tirkolaee2020novel} for sustainable supplier selection. This method integrated Fuzzy ANP, DEMATEL, and TOPSIS, and was embedded in a tri-objective mixed-integer linear programming model, which was validated through a real-world supply chain case study. In \cite{zarepisheh2014generating}, a nonlinear scalarisation method was proposed to generate properly efficient points in multi-objective optimisation problems, with theoretical conditions provided for guaranteeing solution optimality. An adaptive weighted sum test using LASSO regression was developed in \cite{liu2020adaptive} for multi-locus family-based genetic association studies, enhancing detection of both common and rare variants while maintaining statistical rigour.

\vspace{2.0mm}

A Weighted Sum Validity Function optimised using a Hybrid Niching Genetic Algorithm was proposed in \cite{sheng2005weighted} to improve clustering performance by preserving population diversity and incorporating $k$-means hybridisation. A revision to the weighted sum model for robot selection was introduced in \cite{goh1996revised}, where extreme expert inputs were excluded to reduce the effects of outliers, thereby improving decision stability and preventing rank reversal. The analysis and synthesis of weighted-sum (WS) functions in \cite{sasao2006analysis} led to a practical design method using look-up table (LUT) cascades, with applications in digital systems such as bit counting and radix conversion.

\vspace{2.0mm}

The study in \cite{rowley2012aggregating} reviewed multi-criteria decision analysis (MCDA) methods for aggregating sustainability indicators, emphasising the importance of context-specific method selection and the limitations of basic weighted sums. Parameter tuning strategies for weighted ensemble sampling of Markov chains were introduced in \cite{aristoff2020optimizing}, enhancing computational efficiency in steady-state simulation, particularly for rare event estimation. A weighted-sum-based solution for the bi-objective travelling thief problem (BITTP) was proposed in \cite{chagas2022weighted}, where the method outperformed competitors and established new best-known solutions on several benchmark instances.

\vspace{2.0mm}

A portfolio optimisation method tailored for electricity markets was developed in \cite{faia2018genetic} using a genetic algorithm combined with weighted sum scalarisation, producing efficient profit-risk trade-offs with significantly reduced runtime. In \cite{yuan2014efficient}, a reliability-based optimisation (RBO) framework was proposed, which employed weighted importance sampling to separate the reliability and optimisation processes, reducing the need for repeated reliability analyses. An adaptive weighted sum strategy was implemented in \cite{huang2008seeking} for spatial multi-objective optimisation problems, addressing the issue of poor Pareto front coverage in concave regions by iteratively adjusting search direction and refining solution diversity.

\vspace{2.0mm}

A benchmarking study conducted in \cite{Wagner2006} evaluated the performance of several evolutionary multi-objective optimisation algorithms (EMOAs) in many-objective settings. The study assessed traditional Pareto-based methods such as NSGA-II and SPEA2, alongside newer algorithms like $\varepsilon$-MOEA, IBEA, and SMS-EMOA, across test problems with three to six objectives. Results showed that while traditional methods suffer from diminished performance in high-dimensional spaces, newer aggregation- and indicator-based algorithms, such as SMS-EMOA, performed significantly better. A preference-based approach inspired by material science, the Weighted Stress Function Method (WSFM), was introduced in \cite{ferreira2017methodology}, demonstrating improved alignment with decision-maker preferences. A comprehensive survey of many-objective evolutionary algorithms (MaOEAs) was provided in \cite{li2015many}, classifying them into seven key strategic classes and analysing scalability, strengths, and limitations across different problem domains.

\vspace{2.0mm}

A non-Pareto evolutionary optimisation approach called Multiple Single Objective Pareto Sampling (MSOPS) was proposed in \cite{hughes2003multiple}, using weighted min-max strategies to explore high-dimensional Pareto sets, particularly in the presence of complex surfaces. In \cite{ahrari2021weighted}, a weighted pointwise prediction method (WPPM) was developed to address dynamic multi-objective optimisation problems, integrating multi-model forecasts and directional variation strategies to maintain solution diversity and robustness. A weighted maximum (WM) scalarisation approach for multi-objective robot planning was introduced in \cite{Wilde2024}, showing improved coverage of non-convex trade-offs and outperforming traditional weighted sum approaches in diverse motion planning scenarios.

\vspace{2.0mm}

A method for black-box simulation-based multi-objective optimisation using an adaptive weighting scheme was proposed in \cite{deshpande2016multiobjective}, combining the DIRECT algorithm for global search and MADS for local refinement, and dynamically adjusting weights to better approximate the Pareto front. In \cite{chen2023weights}, an empirical comparison of Pareto and weighted search strategies in Search-Based Software Engineering (SBSE) showed that Pareto search consistently outperformed weighted search under sufficient computational budgets, even in the presence of stakeholder preferences. A grid-based local search method called Grid Weighted Sum Pareto Local Search (GWS-PLS) was introduced in \cite{cai2018grid}, combining Pareto dominance and weighted sum strategies to improve computational efficiency, scalability, and diversity in combinatorial optimisation.

\vspace{2.0mm}

A weighted sampling framework for estimating multiple segment-level statistics in large datasets was developed in \cite{cohen2015multi}, enabling efficient computation of f-statistics with statistically guaranteed accuracy from reduced sample sizes. A multi-objective optimisation method based on the Seagull Optimisation Algorithm (MOSOA) was introduced in \cite{dhiman2021mosoa}, which used dynamic archive caching and roulette wheel selection to enhance exploration and exploitation during the search. In \cite{Kashfi2011}, optimisation techniques for VLSI circuit design were presented, addressing conflicting objectives of power dissipation and delay using three scalarisation methods: Weighted Sum (WS), Compromise Programming (CP), and the Satisficing Trade-off Method (STOM), within both convex and non-convex modelling frameworks.

\vspace{2.0mm}

Multi-criteria Polynomial Time Approximation Schemes (PTAS) were developed in \cite{grandoni2014new} for classic $\mathcal{NP}$-hard problems such as spanning tree and bipartite matching under multiple budget constraints. The schemes employed iterative rounding to achieve near-optimal solutions while permitting slight budget violations. An adaptive weighted-sum and clustering-based topology optimisation method was proposed in \cite{ryu2021multi}, producing diverse and well-spaced Pareto-optimal solutions efficiently. The theoretical relationship between decomposition-based scalarisation methods and Pareto-based approaches was analysed in \cite{giagkiozis2015methods}, where Chebyshev scalarisation was shown to yield similar performance under certain assumptions, particularly when search trajectory balance was maintained.

\vspace{2.0mm}

An intelligent sampling approach guided by adaptive weighted-sum methods was proposed in \cite{lin2018intelligent} to solve black-box multi-objective optimisation problems more efficiently, achieving well-distributed Pareto fronts while minimising simulation effort. A weighted sum method incorporating partial preference information was developed in \cite{kaddani2017weighted}, allowing for flexible weight specification and reducing the need for precise preference elicitation in MCDM problems. In \cite{gharehkhani2014extension}, the weighted-sum-of-gray-gases (WSGG) model was extended for simulating radiative heat transfer in gas-soot mixtures, accurately predicting furnace performance across different load conditions.

\vspace{2.0mm}

The Weighted Sum of Segmented Correlation (WSSC) method was introduced in \cite{soor2024weighted} for hyperspectral material identification, using weighted segment-wise correlation indices to enhance the detection of subtle spectral absorption features. This approach demonstrated improved performance over traditional full-spectrum similarity measures. An adaptive weighted sum strategy for simulation-based multi-objective optimisation was proposed in \cite{ryu2009pareto}, termed the Pareto front Approximation with Adaptive Weighted Sum (PAWS), which showed improved convergence and distribution of Pareto-optimal solutions compared to BIMADS. The Integrated Simple Weighted Sum Product (WISP) method, introduced in \cite{stanujkic2021integrated}, combined weighted sum and weighted product techniques to offer a simplified, accessible MCDM tool, validated through comparisons with methods such as TOPSIS and VIKOR.

\vspace{2.0mm}

The Localized Weighted Sum (LWS) method was proposed in \cite{wang2016localized} to enhance many-objective optimisation by combining the efficiency of weighted sum scalarisation with localised search strategies. Operating within hypercones around weight vectors, the method selectively evaluates neighbouring solutions to better manage non-convexity in the objective space. A critical analysis of the Weighted Sum Method (WSM) was provided in \cite{marler2010weighted}, highlighting its conceptual strengths and limitations, especially with respect to preference articulation and performance on non-convex Pareto fronts. The study also offered practical guidelines for selecting weights to mitigate distributional deficiencies in solution quality.

\vspace{2.0mm}

Overall, the literature highlights the versatility and adaptability of WSM in tackling multi-criteria decision problems across various domains, offering insights into its integration with novel optimisation methods, technological advancements, and multi-objective problem-solving.


\section{Weighting the Weighted Sum Method}
Having reviewed various applications associated with the WSM in MCDM, we now turn our attention to the critical task of selecting the weights. Despite its importance, this step is often overlooked or treated superficially in existing literature. We introduce several structured approaches aimed at reducing redundancy, improving computational efficiency, and enabling a comprehensive exploration of nondominated solutions.

\vspace{2.0mm}

Recall that the weighted sum method converts the original problem to \eqref{WSM_Definition}. We focus initially on the bi-objective setting (i.e. $p=2$), before extending th approaches to the case with $p \ge 3$ objectives. 
One natural question explored within this section is how to vary these weights in order to find many (distinct) weakly efficient solutions while additionally avoiding redundancy. It should be noted that such redundancy could occur given that multiple choices of weights could lead to the same solution. Several approaches for varying the weights will be outlined below.

\subsection{\texorpdfstring{Weight Selection for $p=2$ Objectives}{Weight Selection for p = 2 Objectives}}
Note that the bi-objective setting yields 
\begin{equation*}
\min _{\boldsymbol{x} \in \mathcal{X}} \, \sum_{i=1}^2 \lambda_i \, \boldsymbol{c}_i^T\boldsymbol{x} = \lambda_1 \boldsymbol{c}_1^T \boldsymbol{x} + \lambda_2 \boldsymbol{c}_2^T \boldsymbol{x},
\end{equation*}
where $\lambda_1+\lambda_2=1$, $\lambda_1, \lambda_2 \ge 0$ and $\mathcal{X}$ denotes the feasible set of solutions. The method successively varies the weights $\lambda_1$ and $\lambda_2$ in order to find weakly efficient solutions. 

\vspace{2.0mm}

The first approach, which we call the \textit{uniform increment approach}, divides the range $[0,1]$ into $d$ equal subintervals for each weight. Note that each subinterval clearly has length $1/d$. In the bi-objective case, we simply vary the weight $\lambda_1$ from 0 to 1 in increments of $1/d$, while setting $\lambda_2 = 1 - \lambda_1$. The parameter $d$ can be intuitively thought of as the \enquote{depth} of search, where a larger $d$ provides a finer resolution, allowing for more precise sampling at the cost of increased computational effort (and likely greater redundancy). Observe that in this approach we solve $d+1$ problems, namely with weights 
$$
(\lambda_1, \lambda_2) = \big\{ (0,1), (1/d, 1-1/d), \ldots, (1,0) \big\}.
$$ 

\vspace{2.0mm}

The second approach, which we call the \textit{random sampling approach}, instead randomly samples weights from the feasible range, while ensuring that they sum to 1. Since $p = 2$ by assumption, it is sufficient to sample only $\lambda_1$ and set $\lambda_2 = 1 - \lambda_1$. The sampling for $\lambda_1$ could be done from distributions such as the uniform distribution or the beta distribution, allowing for flexibility in controlling the spread and concentration of weights over the feasible range. It should be noted that this method is less systematic, however, it could be useful for higher-dimensional problems. An important observation is that it is unclear when the sampling procedure should terminate to stop searching for nondominated points.

\vspace{2.0mm}

The third approach, which we call the \textit{Latin hypercube sampling (LHS) approach}, informally divides $[0,1]$ into $d$ equal subintervals for each weight. Then if $d$ is even, we randomly sample one value from each interval for each weight, shuffle those sampled values to create multiple combinations, before normalising each combination so that the sum of the (combined) weights equals 1. If instead $d$ is odd, one value is randomly sampled from each interval other than the $\lceil d / 2 \rceil$-th interval in the ordered sequence, from which we sample two values for technical reasons. 

\vspace{2.0mm}

Note that each shuffle creates an ordered sequence of the $d$ (or $d+1$) randomly sampled weights. This process is repeated $s$ times, where $s$ denotes the number of (overall) shuffles, yielding an ordered sequence of $sd$ (or $s(d+1)$) weights. Adjacent weights in the sequence are then paired to create $sd/2$ (or $s(d+1)/2$) combinations. Each combination is subsequently normalised such that the sum of the (combined) weights equals 1.

\vspace{2.0mm}
 
This approach interestingly leads to the normalised weights being roughly normally distributed, with most values concentrated around the mean of 0.5, rather than near 0 or 1. This is since the process draws from equally spaced intervals across $[0,1]$, with fewer values coming from the extreme ends near 0 and 1. For instance, to generate a combination where $\lambda_1$ is near 0, we would need one sample from an interval close to 0 and another from an interval close to 1 in the same combination, a statistically less likely event. Thus, the weights are much more likely to be close to 0.5 rather than the extremes. This weight concentration may therefore not result in a good spread of solutions in an optimisation context. 

\vspace{2.0mm}

\begin{figure}[ht!] 
\centering
    \begin{subfigure}{0.48\linewidth}
        \includegraphics[width=\linewidth]{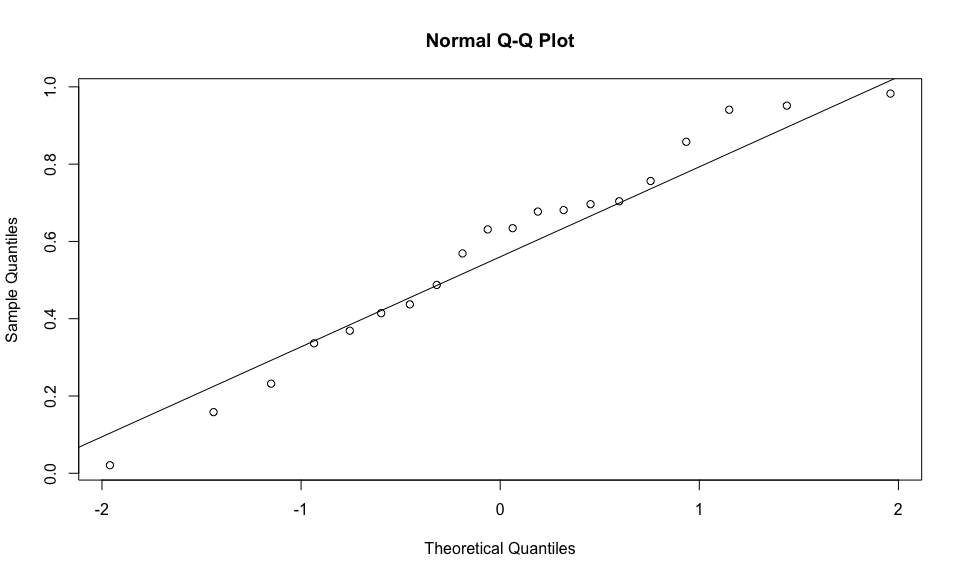}
        \caption{Normal Q-Q plot with $d=20$ and $s=2$.}
        \label{figsub:1}
    \end{subfigure}
    \hfill
    \begin{subfigure}{0.48\linewidth}
        \includegraphics[width=\linewidth]{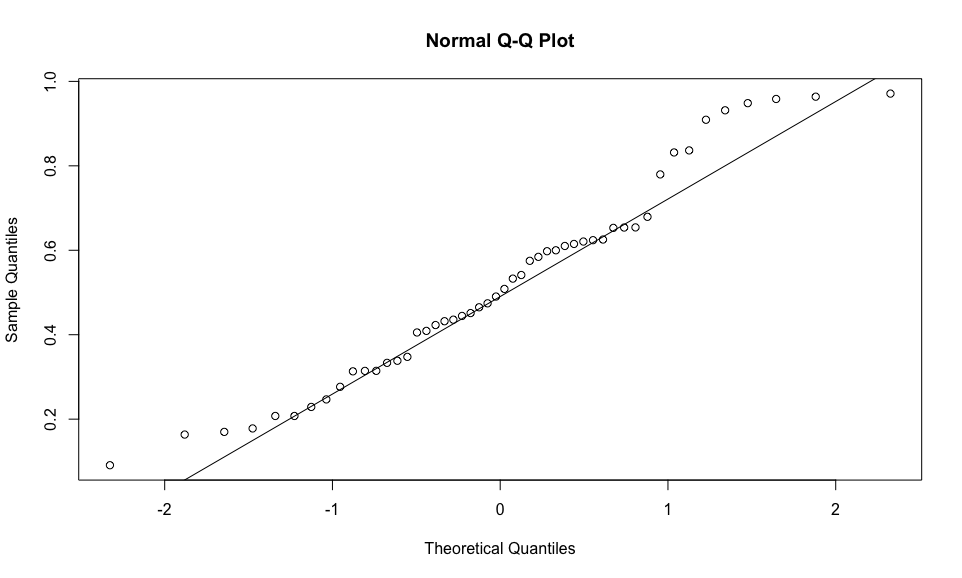}
        \caption{Normal Q-Q plot with $d=20$ and $s=5$.}
        \label{figsub:2}
    \end{subfigure}
\hfill
    \begin{subfigure}{0.48\linewidth}
        \includegraphics[width=\linewidth]{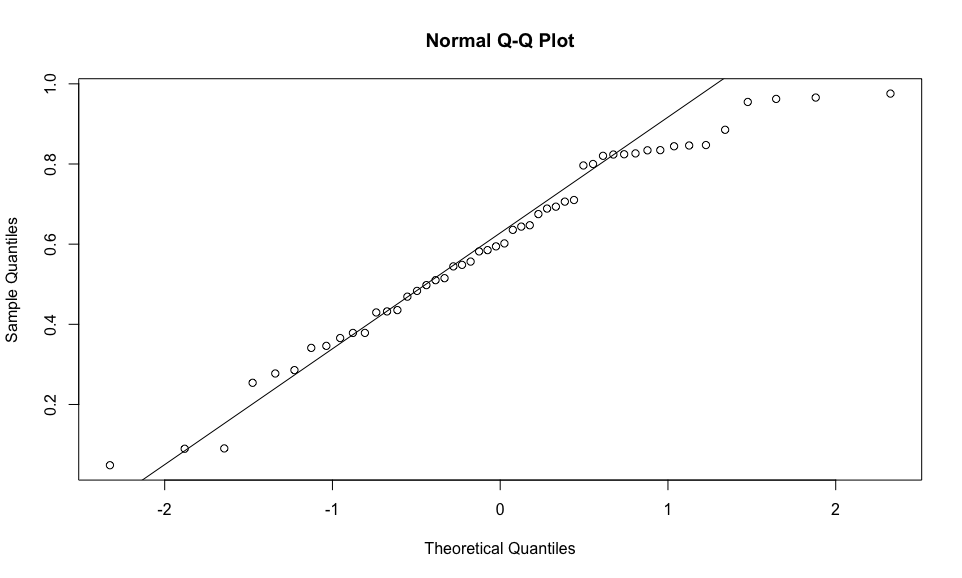}
        \caption{Normal Q-Q plot with $d=50$ and $s=2$.}
        \label{figsub:3}
    \end{subfigure}
    \hfill
    \begin{subfigure}{0.48\linewidth}
        \includegraphics[width=\linewidth]{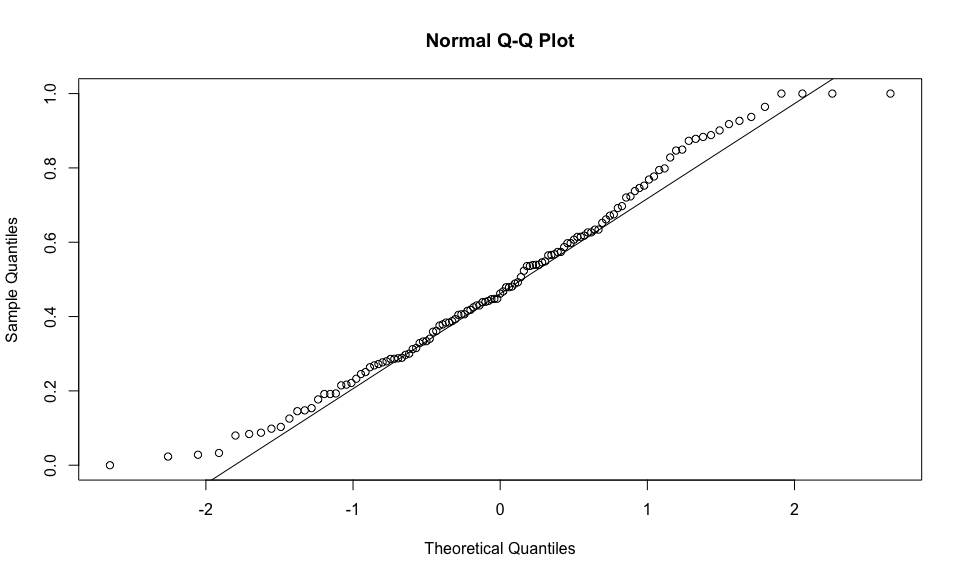}
        \caption{Normal Q-Q plot with $d=50$ and $s=5$.}
        \label{figsub:4}
    \end{subfigure}
    
    \caption{This figure presents four normal Q-Q plots for different (smaller) values of $d$ (number of intervals) and $s$ (number of shuffles) in the Latin hypercube sampling (LHS) approach.}
    \label{fig:1234}
\end{figure}

This claim is further supported by presenting several normal quantile-quantile (Q-Q) plots (namely Figures \ref{fig:1234} and \ref{fig:1234_larger}) of normalised weights drawn from $d$ intervals and shuffled $s$ times, for selected even values of $d$ and selected $s$. Recall that the process yields $sd/2$ combinations. Then, from each combination, we randomly select one value for plotting, resulting in $sd/2$ points on each plot. Note that in each normal Q-Q plot, for each point $(x,y)$, $x$ corresponds to one of the $sd/2$ quantiles from a normal distribution, and $y$ corresponds to one of the $sd/2$ weights. When the points lie close to the line $y = x$, this indicates that the values are approximately normally distributed.

\vspace{2.0mm}

Furthermore, at the tails of the normal Q-Q plots (Figures \ref{fig:1234} and \ref{fig:1234_larger}), we observe deviations from the line. In particular, smaller $x$-values (representing the lower tail) tend to lie above the line, indicating that the sampled values are larger than expected, suggesting a lighter left tail. Conversely, larger $x$-values (representing the upper tail) tend to lie below the line, indicating that the sampled values are smaller than expected, implying a lighter right tail. This suggests that the distribution of weights has thinner tails compared to a normal distribution.

\vspace{2.0mm}

The concentration of weights in the LHS approach is problematic, as we would ideally like our weights to be more evenly distributed in order to explore a broader range of potential nondominated points. This motivates us to introduce additional structure to the LHS approach, leading to the development of a more refined approach. 

\vspace{2.0mm}

The fourth approach, which we call the \textit{structured Latin hypercube sampling (SLHS) approach}, works similarly to the LHS approach, however, before sampling the $d$ intervals are structured. In particular, firstly the approach once more divides $[0,1]$ into $d$ equal subintervals. Then assuming $d$ is even, we pair all intervals $[a_1, a_2]$ and $[b_1, b_2]$ such that 
\begin{equation} \label{SLHS_property}
a_1 + b_2 = a_2 + b_1 = 1.
\end{equation}
Observe that given $d$ is by assumption even, we will clearly have $d/2$ pairs of intervals satisfying this property. Then we randomly sample one value from each interval for each weight. Next, we form $d/2$ pairs using the sampled values from the matched intervals. Finally, we normalise each combination such that the sum of the weights equals 1. The following example illustrates the SLHS approach with a small number of intervals. 

\vspace{2.0mm}

\begin{example}
Suppose that $p=2$ and $d=4$. We firstly divide the range $[0,1]$ into the intervals
$$
\left[0, \frac{1}{4}\right], \left[\frac{1}{4}, \frac{1}{2}\right], \left[\frac{1}{2}, \frac{3}{4}\right], \left[\frac{3}{4}, 1\right]
$$
Next, we pair intervals satisfying \eqref{SLHS_property}, which yields
\begin{itemize}
    \item Pair 1:$[0, \, 1/4]$ and $[3/4, \, 1]$, and
    \item Pair 2: $[1/4, \, 1/2]$ and $[1/2, \, 3/4]$.
\end{itemize}
We then randomly sample one value from each interval, for example
\begin{itemize}
    \item Pair 1: Sample 0.06637717 from $[0, \, 1/4]$ and 0.843031 from $[3/4, \, 1]$, and
    \item Pair 2: Sample from 0.3932133 $[1/4, \, 1/2]$ and 0.7270519 from $[1/2, \, 3/4]$.
\end{itemize}
Finally, we normalise each pair such that their sum equals 1, which (upon rounding to three decimal places) gives
\begin{itemize}
    \item Pair 1: Normalised weights are 
    $$
    \frac{0.06637717}{0.06637717+0.843031} = 0.073
    $$
    and 
    $$
    \frac{0.843031}{0.06637717+0.843031} = 0.927, \text{ and}
    $$
    \item Pair 2: Normalised weights are 
    $$
    \frac{0.3932133}{0.3932133+0.7270519} = 0.3510
    $$
    and 
    $$
    \frac{0.7270519}{0.3932133+0.7270519} = 0.6490.
    $$
\end{itemize}
Thus, two structured, distinct weight pairs are generated, namely 
$$
(0.073, \, 0.927) \quad \text{and} \quad
(0.3510, \, 0.6490).
$$
\end{example}

\vspace{2.0mm}

This approach can be thought of as the structured analogue to the LHS approach, given that it is guaranteed that the sampled pairs after normalisation must remain within their corresponding initial (paired) intervals, which is formally proven below (Lemma \ref{Lemma_post_normalise}). Note that if instead $d$ is odd, the $\lceil d/2 \rceil$-th interval lacks a natural partner under the pairing scheme. To accommodate this, we simply sample two values from this interval and proceed with normalisation similarly, thereby preserving the pairing structure and theoretical guarantees of the method.


\begin{lemma} \label{Lemma_post_normalise}
Suppose $[a_1, a_2]$ and $[b_1, b_2]$ are (ordered) subintervals of $[0,1]$ with 
$$
a_1 < a_2 < b_1 < b_2
$$ 
satisfying 
$$a_1 + b_2 = a_2 + b_1 = 1.$$ 
Then any $a \in [a_1, a_2]$ and $b \in [b_1, b_2]$ satisfy 
\begin{equation} \label{bounding_lemma}
\frac{a}{a+b} \in [a_1, a_2] \quad \text{ and } \quad \frac{b}{a+b} \in [b_1, b_2].
\end{equation}
\end{lemma}

\begin{proof}
Suppose for contradiction that \eqref{bounding_lemma} does not hold, i.e. 
$$
\frac{a}{a+b} \notin [a_1, a_2]
$$ 
or 
$$
\frac{b}{a+b} \notin [b_1, b_2].$$ 
Note that this is the case when $a/(a+b) < a_1$, $a/(a+b) > a_2$, $b/(a+b) < b_1$ or $b/(a+b) > b_2$ hold. Each such case will be considered in turn.

\vspace{2.0mm}

Firstly, suppose that $a/(a+b) < a_1$ holds. This yields 
$$
a(1-a_1) < a_1 b
$$ 
through algebraic manipulation. Observe that if $a_1 = 0$, then we deduce that $a < 0$, which is clearly a contradiction. If instead $a_1 \ne 0$, then upon dividing by $a_1$ we have 
$$\frac{a}{a_1} \, (1-a_1) < b.$$ 
Note that $a \ge a_1$ by assumption and hence ${a}/{a_1} \ge 1$. In particular, this implies that $b > 1-a_1$ holds, which is a contradiction given $b \le b_2$ and $b_2 = 1 -a_1$.

\vspace{2.0mm}

Secondly, suppose that $a/(a+b) > a_2$ holds. This yields 
$$
a(1-a_2) > a_2 b.
$$ 
Observe that the assumed ordering $0 \le a_1 < a_2 < b_1 < b_2 \le 1$ implies that $a_2 \ne 0$. Upon dividing by $a_2$ we have 
$$
\frac{a}{a_2} \, (1-a_2) > b.
$$ 
Note that $a \le a_2$ by assumption and hence $a/a_2 \le 1$. It follows that $1-a_2 > b$ holds, which is a contradiction given that $b \ge b_1$ and $b_1 = 1 -a_2$. 

\vspace{2.0mm}

Thirdly, suppose that $b/(a+b) < b_1$ holds. This yields 
$$
b(1-b_1) < b_1 a
$$ 
via manipulation. Note that the assumed ordering implies that $b_1 \ne 0$. Upon dividing by $b_1$ we yield 
$$
b/b_1 \, (1-b_1) < a.
$$ 
Note that $b \ge b_1$ by assumption and hence $b/b_1 \ge 1$. This implies that $1-b_1 < a$ holds, which is a contradiction given that $a \le a_2$ and $a_2 = 1 - b_1$.

\vspace{2.0mm}

Finally, suppose that $b/(a+b) > b_2$ holds. This yields 
$$
b(1-b_2) > b_2 a.
$$ 
Note that $b_2 \ne 0$ and then upon dividing by $b_2$ we deduce that 
$$
\frac{b}{b_2} \, (1-b_2) > a.
$$ 
Note that $b \le b_2$ and hence $b / b_2 \le 1$. This implies that $1-b_2 > a$ holds, which is a contradiction given that $a \ge a_1$ and $a_1 = 1 - b_2$. In particular, the four cases considered demonstrate that \eqref{bounding_lemma} holds, which completes the proof as required. 
\end{proof}

\vspace{2.0mm}

\begin{figure}[ht!] 
\centering
    \begin{subfigure}{0.49\linewidth}
        \includegraphics[width=\linewidth]{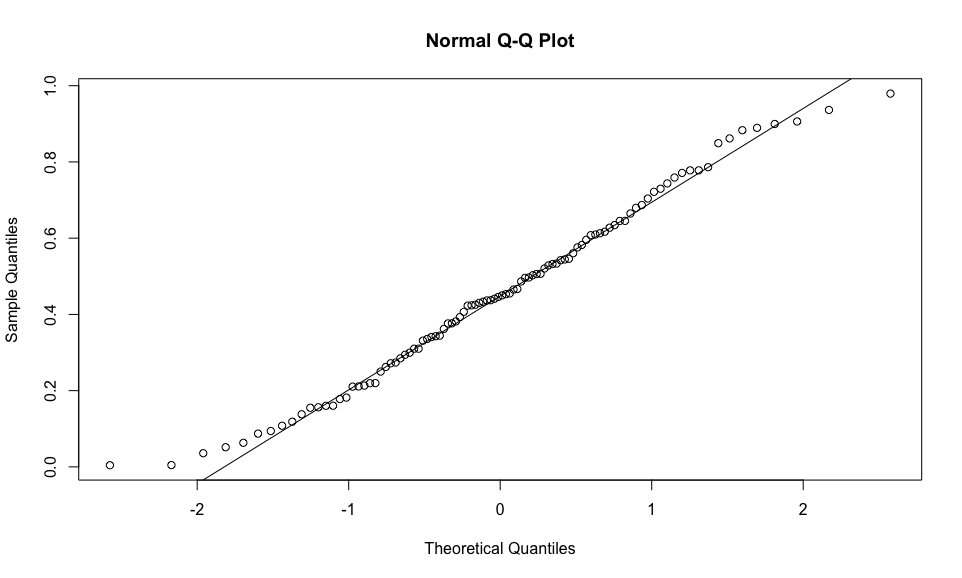}
        \caption{Normal Q-Q plot with $d=100$ and $s=2$.}
        \label{figsub:11}
    \end{subfigure}
    \hfill
    \begin{subfigure}{0.49\linewidth}
        \includegraphics[width=\linewidth]{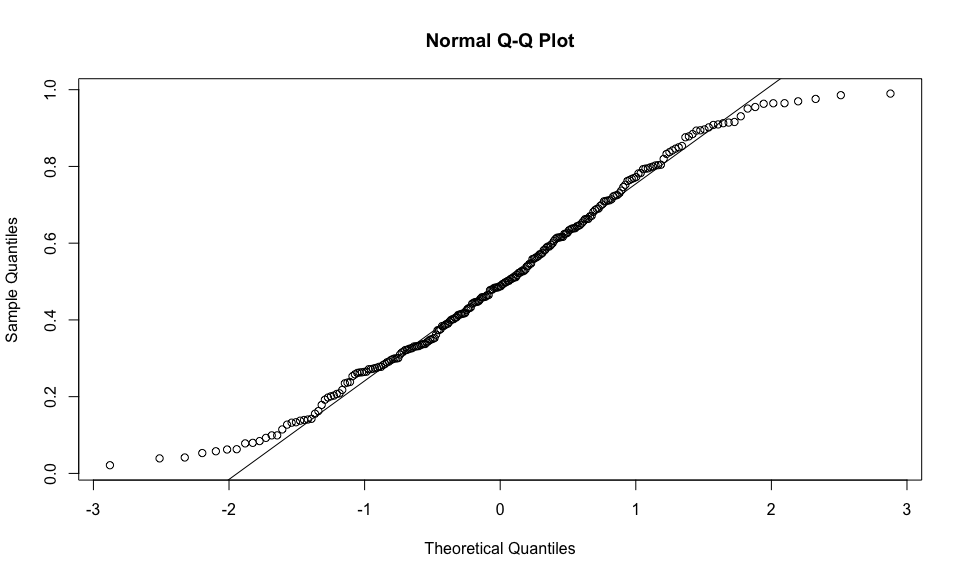}
        \caption{Normal Q-Q plot with $d=100$ and $s=5$.}
        \label{figsub:21}
    \end{subfigure}
\hfill
    \begin{subfigure}{0.49\linewidth}
        \includegraphics[width=\linewidth]{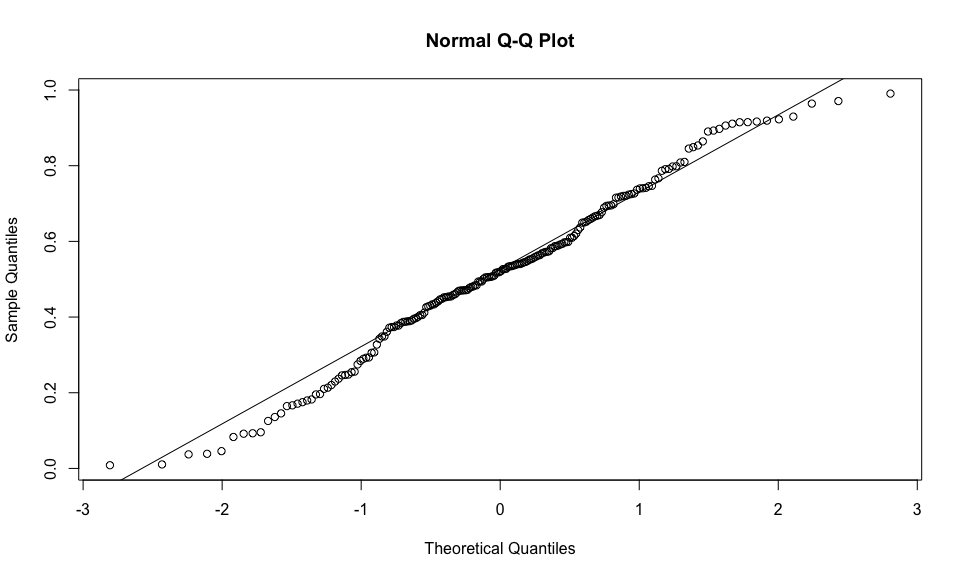}
        \caption{Normal Q-Q plot with $d=200$ and $s=2$.}
        \label{figsub:31}
    \end{subfigure}
    \hfill
    \begin{subfigure}{0.49\linewidth}
        \includegraphics[width=\linewidth]{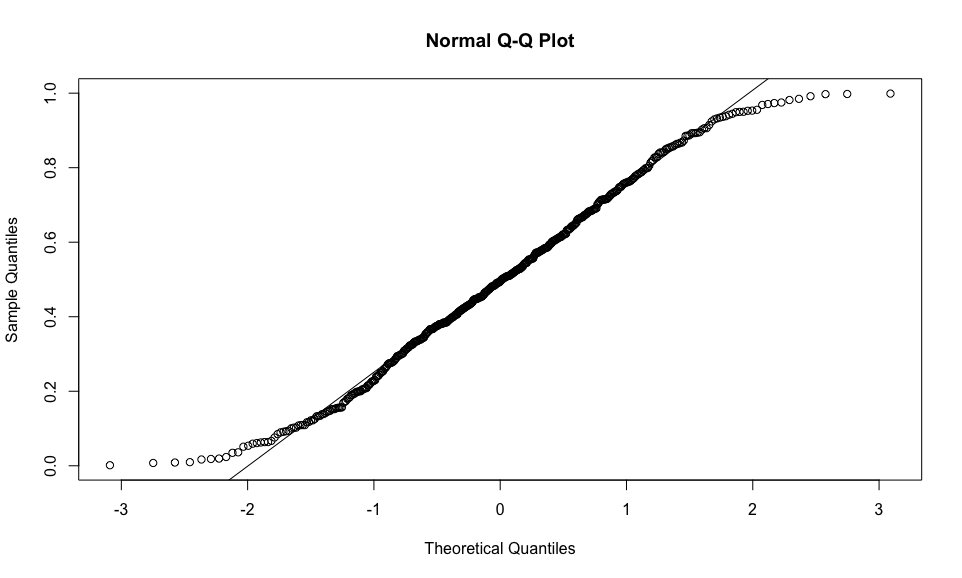}
        \caption{Normal Q-Q plot with $d=200$ and $s=5$.}
        \label{figsub:41}
    \end{subfigure}
    
    \caption{This figure presents four normal Q-Q plots for different (larger) values of $d$ (number of intervals) and $s$ (number of shuffles) in the Latin hypercube sampling (LHS) approach.}
    \label{fig:1234_larger}
\end{figure}

\vspace{2.0mm}

\vspace{2.0mm}

The fifth approach, which we call the \textit{structured adaptive approach}, starts with an initial set of weights before adapting them as required before terminating searches based on the solutions found. The method informally identifies \enquote{large gaps} between nondominated points (in terms of some distance metric) and then subdivides those regions to introduce finer resolution. In particular, we initially follow the uniform increment approach by dividing $[0,1]$ into $d$ equal subintervals, before then adapting our weights by subdividing (some of) our $d$ subintervals into $d$ equal subintervals and repeating until termination. Note that the decision regarding if a subinterval will be further subdivided will depend on the previously found nondominated points that correspond to the end points of the subinterval being considered and toleration threshold and redundancy bounding parameters $\tau\ge0$ and $\rho \in [0,1]$, respectively. 

\vspace{2.0mm}

Being more precise, suppose that we are considering if we should subdivide the interval $[a_1, a_2]$ as described above and that the corresponding weight pairs, namely $(a_1, 1-a_1)$ and $(a_2, 1-a_2)$, yield the nondominated points $\boldsymbol{n}_1$ and $\boldsymbol{n}_2$, respectively. In particular, one should subdivide only if 
$$
\|\boldsymbol{n}_1-\boldsymbol{n}_2\|_2 > \tau,
$$
where  $\| \cdot \|_2$ denotes the $\ell_2$-norm. Further, the overall approach should terminate if
$\mathcal{N}/\mathcal{D} < \rho$ 
holds, where $\mathcal{N}$ and $\mathcal{D}$ denote the total number of distinct nondominated points found and the total number of intervals searched, respectively. It should be noted that the second condition enforces termination when the overall level of redundancy has grown significantly. Observe that the stopping conditions introduced here mean that the running time of such an implementation will be intrinsically related to $\tau$ and $\rho$ and, as such, it is natural to consider what is a suitable \enquote{size} of $\tau$ and $\rho$ in order to avoid redundancy wherever possible. 

\subsection{\texorpdfstring{Weight Selection for $p \ge 3$ Objectives}{Weight Selection for p ≥ 3 Objectives}}
Notice that the setting with $p \ge 3$ objectives converts the original problem to \eqref{WSM_Definition}, where $\sum_{i=1}^p \lambda_i=1$ and $\lambda_i \ge 0$ for all $i \in \{1,2,\ldots,p\}$. The method once more successively varies the weights $\lambda_i$ in order to find weakly efficient solutions. It should be emphasised that while the weight selection methods introduced for bi-objective problems rely heavily on pairings and simple linear complements, extending these methods to cases with three or more objectives becomes inherently more complex. This increased complexity arises since the weight vectors must now satisfy a higher-dimensional simplex constraint, requiring the development of more generalised sampling strategies that effectively explore this space. 

\vspace{2.0mm}

The first approach, namely the uniform increment approach, similarly divides the range $[0,1]$ into $d$ equal subintervals for each weight. Here each weight $\lambda_i$ can take values from 
$\{0, 1/d, \ldots, 1\}$. Then we generate all possible combinations of weights $\{ \lambda_1, \lambda_2, \ldots, \lambda_p \}$ such that $\sum_{k=1}^p \lambda_i=1$ and $\lambda_i \ge 0$ for all $i \in \{1,2,\ldots,p\}$. This can be done systematically to ensure normalisation. The natural approach is to make use of a nested loop to iterate over all combinations of $\lambda_1, \lambda_2, \ldots, \lambda_{p-1}$, before calculating the remaining weight $\lambda_p$ to ensure that their sum is 1, where any combination with $\lambda_p \not\in [0,1]$ is discarded. This results in solving at most $(d+1)^{p-1}$ problems, as exemplified below. 

\vspace{2.0mm}

\begin{example}
Suppose that $p=3$ and $d=2$. Observe that $\lambda_1, \lambda_2, \lambda_3 \in \{0, \frac{1}2, 1\}$ in such case. Then we iterate over $\lambda_1$ and $\lambda_2$ before calculating $\lambda_3 = 1 - (\lambda_1+\lambda_2)$, namely:
\begin{enumerate}[style=unboxed, font=\normalfont]
    \item[1)] for $\lambda_1 = 0$:
        \begin{itemize}
                \item $\lambda_2 = 0$ and $\lambda_3 = 1 - (0+0) = 1$,
                \item $\lambda_2 = \frac{1}{2}$ and $\lambda_3 = 1 - (0+\frac{1}{2}) = \frac{1}{2}$, and
                \item $\lambda_2 = 1$ and $\lambda_3 = 1 - (0+1) = 0$.
        \end{itemize}
    \item[2)] for $\lambda_1 = \frac{1}{2}$:
        \begin{itemize}
                \item $\lambda_2 = 0$ and $\lambda_3 = 1 - (\frac{1}{2}+0) = \frac{1}{2}$,
                \item $\lambda_2 = \frac{1}{2}$ and $\lambda_3 = 1 - (\frac{1}{2}+\frac{1}{2}) = 0$, and
                \item $\lambda_2 = 1$ and $\lambda_3 = 1 - (\frac{1}{2}+1) = -\frac{1}{2}$, which is invalid as $\lambda_3 < 0$.
        \end{itemize}
    \item[3)] for $\lambda_1 = 1$:
        \begin{itemize}
                \item $\lambda_2 = 0$ and $\lambda_3 = 1 - (1+0) = 0$,
                \item $\lambda_2 = \frac{1}{2}$ and $\lambda_3 = 1 - (1+\frac{1}{2}) = -\frac{1}{2}$, which is invalid as $\lambda_3 < 0$, and
                \item $\lambda_2 = 1$ and $\lambda_3 = 1 - (1+1) = -1$, which is invalid as $\lambda_3 < 0$.
        \end{itemize}
\end{enumerate}
Thus, the valid combinations are
$$
\begin{aligned}
\Big\{ &\left(0,0,1\right), \left(0,\frac{1}{2},\frac{1}{2}\right), \left(0,1,0\right), \\
&\left(\frac{1}{2},0,\frac{1}{2}\right), \left(\frac{1}{2},\frac{1}{2},0\right), \left(1,0,0\right) \Big\}.
\end{aligned}
$$
\end{example}

It should be noted that the aforementioned upper bound of $(d+1)^{p-1}$ can be refined to instead state an explicit closed formula for the number of valid combinations. The number of valid combinations is
$$
\binom{d + p - 1}{p - 1} = \frac{(d+p-1)!}{(p-1)! \, d!},
$$
which is the binomial coefficient for integers $d+p-1$ and $p-1$, respectively. This follows immediately in light of the stars and bars theorem (or formula) (see e.g. \cite[Chapter 3]{beeler2015count}), which provides a way to count the number of nonnegative solutions to the equation 
$$
\lambda_1 + \lambda_2 + \cdots + \lambda_p = d.
$$

\vspace{2.0mm}

The rapid growth of this binomial coefficient as both $d$ and $p$ increase is demonstrated in Figure \ref{fig:binomial_growth}, which illustrates how the number of valid combinations grows exponentially as we vary $d$ (depth) and $p$ (number of objectives). The figure particularly highlights the steep growth rate for larger $p$, reflecting the complexity and computational effort required to generate and evaluate combinations as the dimensionality increases.

\vspace{2.0mm}

\begin{figure}[ht!] 
\centering
        \includegraphics[width=0.8\linewidth]{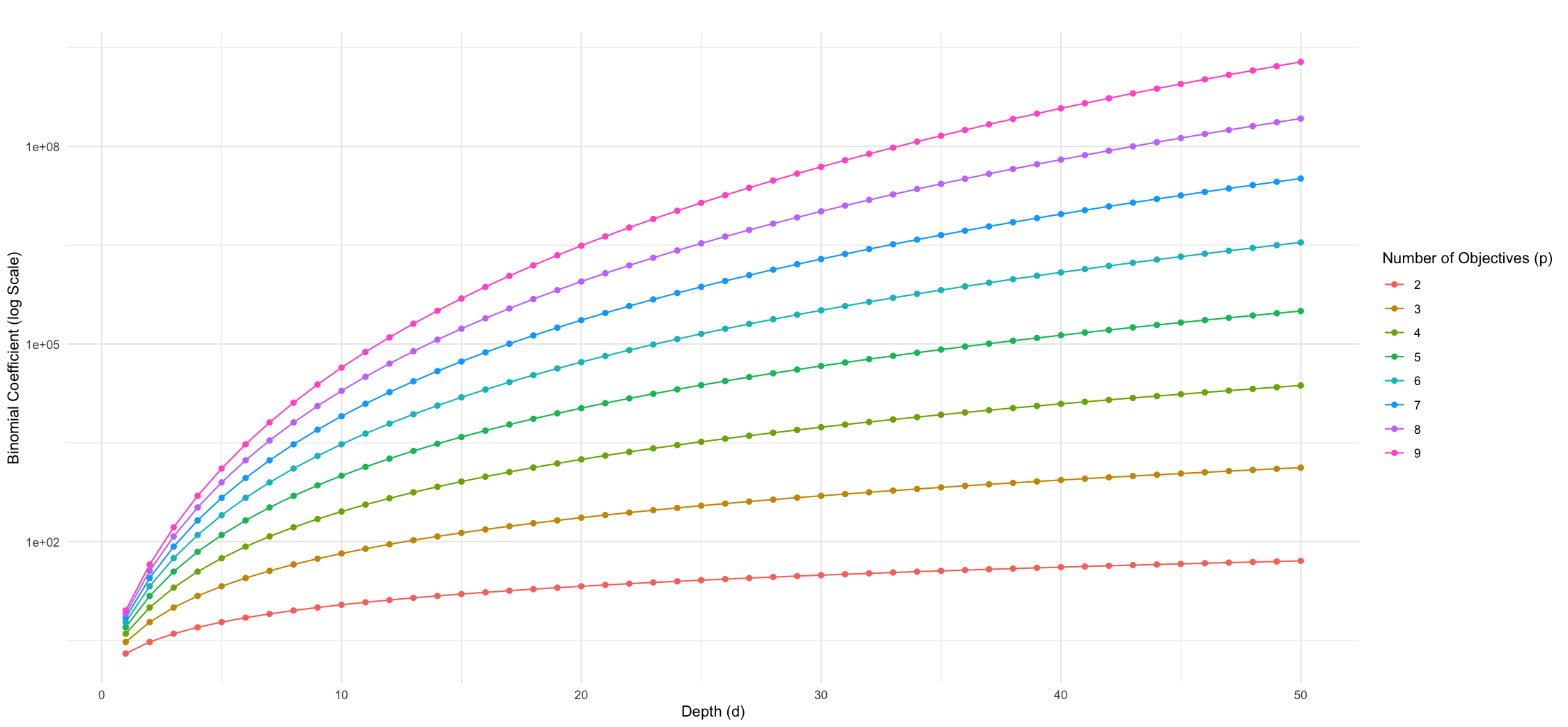}
    \caption{This figure illustrates the growth of the binomial coefficient $\binom{d + p - 1}{p - 1}$ on a logarithmic scale, as $d$ (depth) and $p$ (number of objectives) vary. The colours denote different values of $p$.}
    \label{fig:binomial_growth}
\end{figure}

\vspace{2.0mm}

The second approach, namely the random sampling approach, can be extended as follows. Since $p \ge 3$ by assumption, it is no longer to sufficient to sample only one $\lambda_i$. One natural extension relies on the Dirichlet distribution, which is the multivariate generalisation of the beta distribution. The Dirichlet distribution (of order $K \ge 2$) with parameters $\alpha_1, \alpha_2, \ldots, \alpha_K >0$ (which are denoted for convenience by $\boldsymbol{\alpha}$) has probability density function with respect to the $(K-1)$-th dimensional volume or Lebesgue measure (see e.g. \cite[Chapter 13]{bartle1995elements}) given by 
$$
f(x_1, x_2, \ldots, x_K ; \alpha_1, \alpha_2, \ldots, \alpha_K) = \frac{1}{B(\boldsymbol{\alpha})} \prod_{i=1}^K x_i^{\alpha_i-1}, 
$$
where $\sum_{i=1}^K x_i=1$ and $x_i \in [0,1]$ for all $i \in \{1,2,\ldots,K\}$, i.e. that $x_1, x_2, \ldots, x_K$ belong to the unit (or standard) $(K-1)$-dimensional simplex in $\mathbb{R}^K$, as illustrated in Figure \ref{fig:unit_simplex_R3}. The normalising constant of the Dirichlet distribution is the multivariate beta function, namely
$$
B(\boldsymbol{\alpha})=\frac{\displaystyle\prod_{i=1}^K \Gamma\left(\alpha_i\right)}{\Gamma\big(\displaystyle\sum_{i=1}^K \alpha_i\big)},
$$
where $\Gamma(\cdot)$ denotes the gamma function. 

\begin{figure}[ht!] 
\centering
        \includegraphics[width=0.8\linewidth]{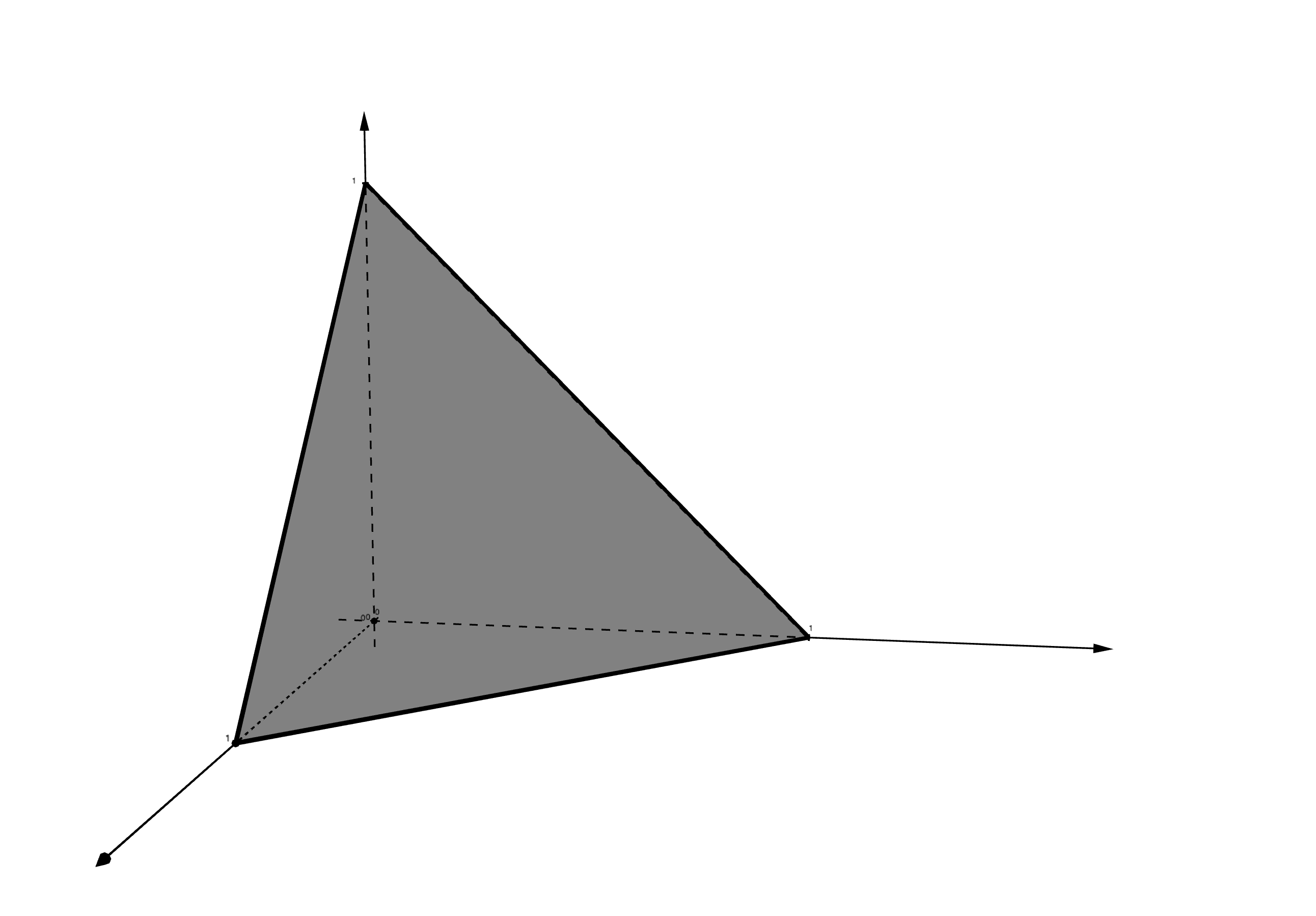}
        \label{fig:structured_adaptive_tuning_Figure2}
    \caption{This figure illustrates the 2-dimensional unit (or standard) simplex in $\mathbb{R}^3$, whose vertices are the 3 standard unit vectors in $\mathbb{R}^3$.}
    \label{fig:unit_simplex_R3}
\end{figure}

\vspace{2.0mm}

Note for completeness that the parameters $\alpha_i$ in the Dirichlet distribution play a crucial role in shaping the distribution of weights. Being more specific, each $\alpha_i$ for $i \in \{1,2,\ldots,K\}$ controls the expected concentration of $x_i$ within the unit simplex. Suppose, for simplicity, that $\boldsymbol{\alpha}$ is symmetric. If $\alpha_i > 1$, then the distribution tends to favour larger values of $x_i$, geometrically pushing the probability mass toward the centre of the unit simplex. If $\alpha_i < 1$, then the distribution tends to favour smaller values of $x_i$, concentrating the probability mass around the facets and vertices of the simplex. If $\alpha_i = 1$ for all $i \in \{1,2,\ldots,K\}$, then the distribution becomes uniform over the simplex. This is illustrated in Figure \ref{fig:Dir_Dist}. Thus, by carefully adjusting the parameters $\alpha_i$, it is possible to control how the weights are distributed while ensuring $\sum_{i=1}^K x_i=1$ and $x_i \in [0,1]$ for each $i$ hold as required.

\begin{figure}[ht!]
\centering
    \begin{subfigure}{0.49\linewidth}
        \includegraphics[width=\linewidth]{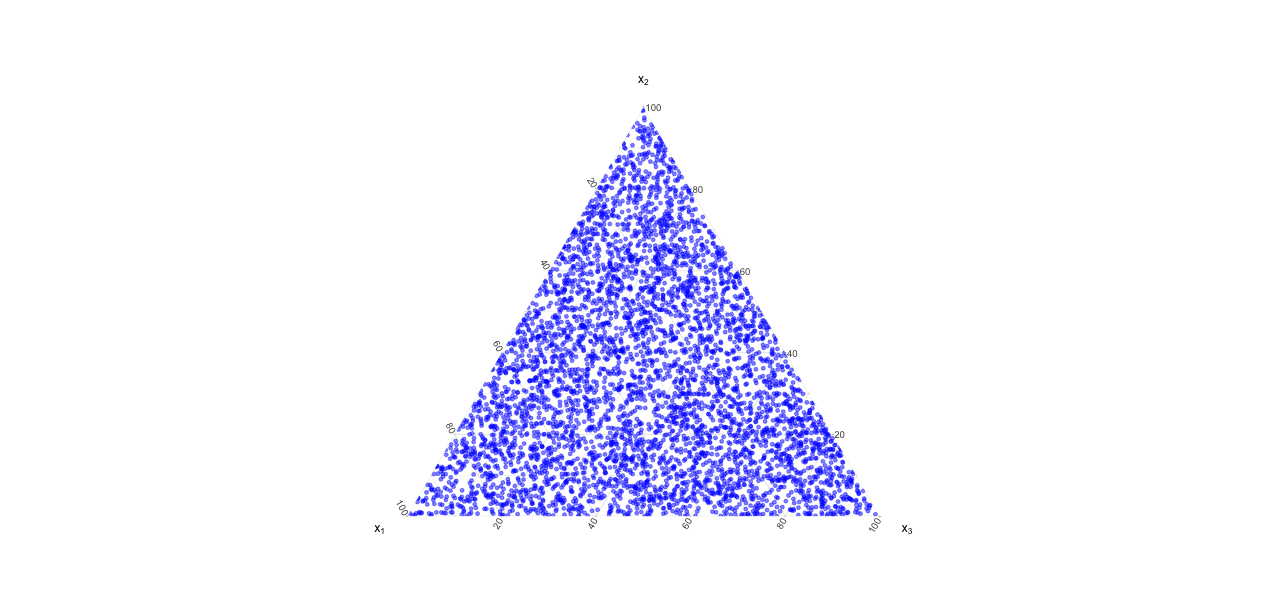}
        \caption{Dirichlet distribution with $\boldsymbol{\alpha} = (1,1,1)$}
        \label{fig:Dir_Dist1}
    \end{subfigure}
    \hfill
    \begin{subfigure}{0.49\linewidth}
        \includegraphics[width=\linewidth]{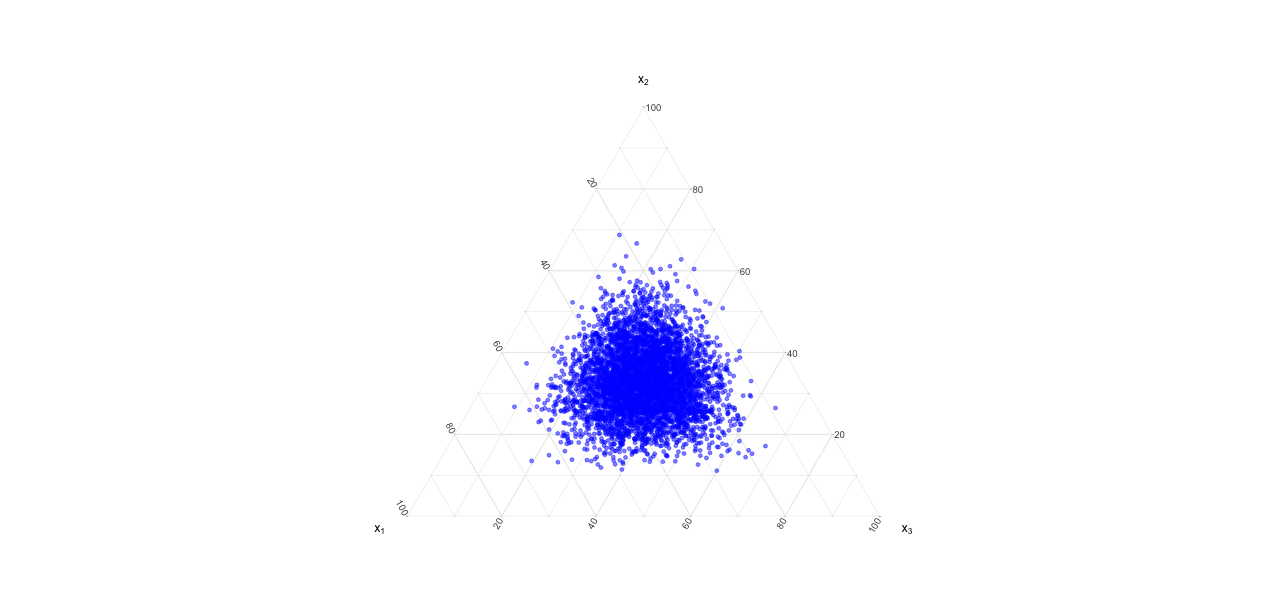}
        \caption{Dirichlet distribution with $\boldsymbol{\alpha} = (10,10,10)$}
        \label{fig:Dir_Dist2}
    \end{subfigure}

    \vspace{0.5em} 

    \begin{subfigure}{0.49\linewidth}
        \includegraphics[width=\linewidth]{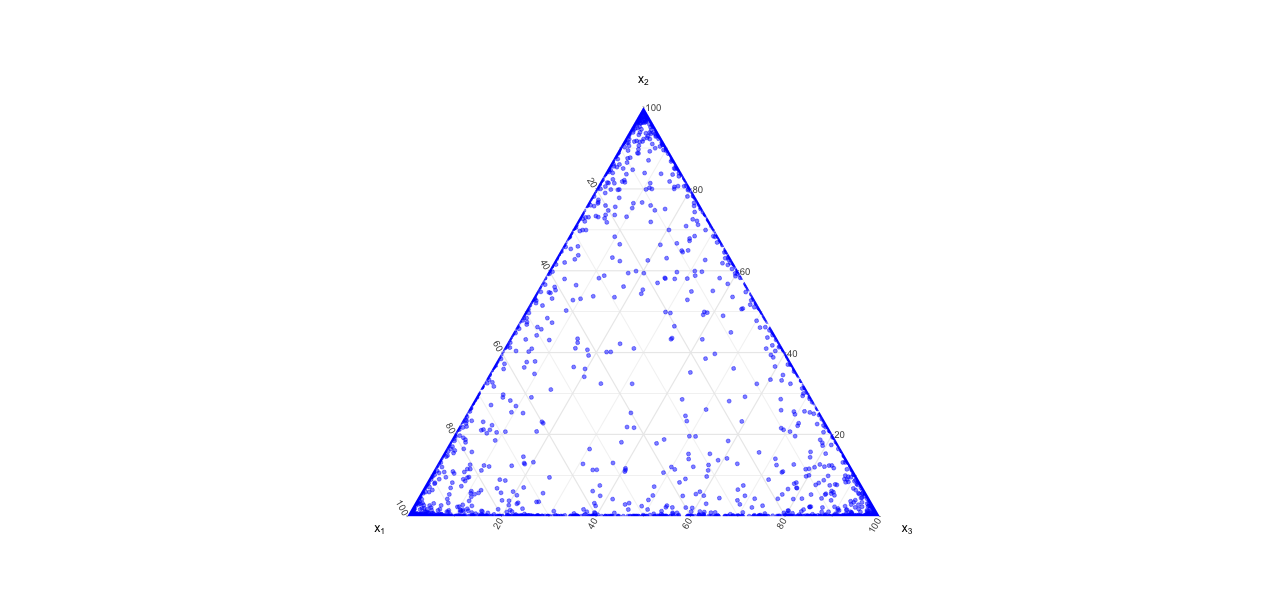}
        \caption{Dirichlet distribution with $\boldsymbol{\alpha} = (0.1,0.1,0.1)$}
        \label{fig:Dir_Dist3}
    \end{subfigure}
    \hfill
    \begin{subfigure}{0.49\linewidth}
        \includegraphics[width=\linewidth]{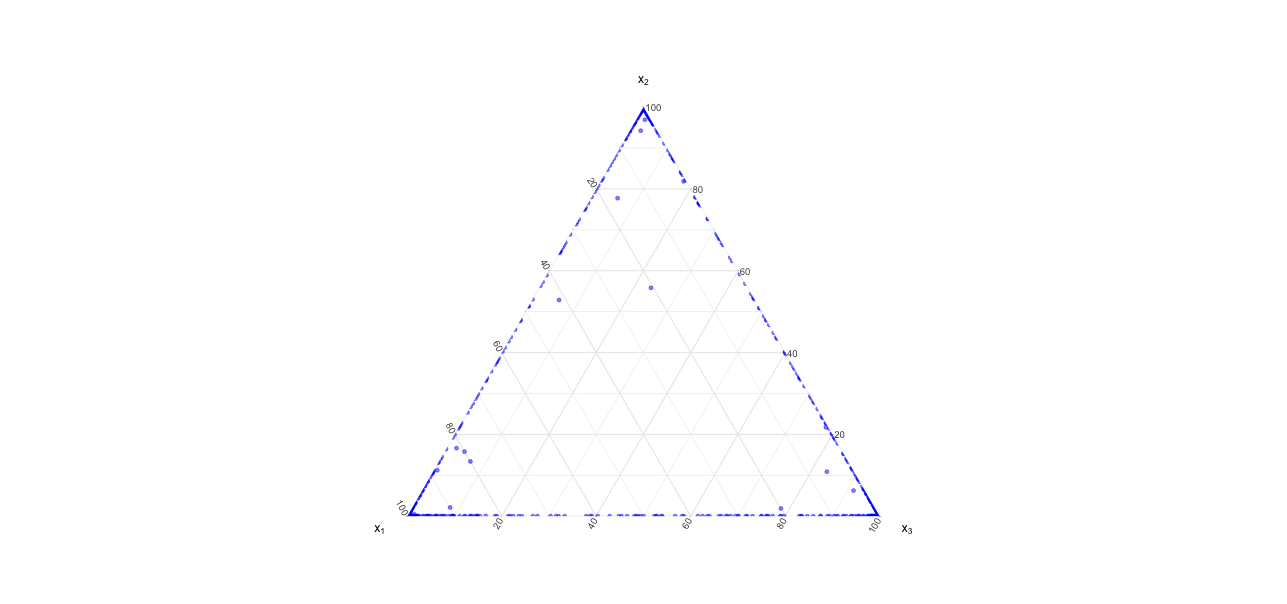}
        \caption{Dirichlet distribution with $\boldsymbol{\alpha} = (0.01,0.01,0.01)$}
        \label{fig:Dir_Dist4}
    \end{subfigure}
    
    \caption{This figure illustrates the Dirichlet distribution for symmetric $\boldsymbol{\alpha} = (\alpha_1, \alpha_2, \alpha_3)$ with $K=3$. Note that each plot features 5000 samples, where each point represents a sample consisting of three nonnegative components that sum to one.}
    \label{fig:Dir_Dist}
\end{figure}

\vspace{2.0mm}

The third approach, namely the LHS approach, can be extended as follows. The approach begins by similarly dividing the range $[0,1]$ into $d$ equal subintervals for each weight. Then we randomly sample one value from each interval for each weight, shuffle these samples to create combinations before normalising each combination such that the sum of the combined weights is 1. 

\vspace{2.0mm}

The fourth approach, namely the SLHS approach, can be extended as follows. The approach begins by similarly dividing the range $[0,1]$ into $d$ equal subintervals for each weight
, as in the LHS approach. Thus, we yield the $d$ subintervals 
\begin{equation} \label{0_1_d_division}
\left[0, \frac{1}{d}\right], \left[\frac{1}{d}, \frac{2}{d}\right], \ldots, \left[\frac{d-1}{d}, 1\right]
\end{equation}
with (set theoretic) union is $[0,1]$. 

\vspace{2.0mm}

Denote by $m_i \in [0,1]$ the midpoint of the subinterval $[a_i, b_i]$, i.e. 
$m_i = (a_i + b_i)/2$. Intuitively, we aim to select $p$ subintervals such that the sum of their midpoints is \enquote{close} to 1. We then collect samples from the subintervals, form $p$-element tuples of the sampled values from the selected subintervals, and finally normalise each combination such that the sum of the (normalised) weights is 1. To measure this \enquote{closeness}, let $\delta \ge 0$ be a parameter representing the allowable deviation of the sum of midpoints from 1. 

\vspace{2.0mm}

Thus, we require the selected subintervals to satisfy the inequalities
\begin{equation} \label{delta_upper_lower}
1 - \delta \le \sum_{i=1}^p m_i \le 1 + \delta, 
\end{equation}
where, to simplify notation, $m_1, m_2, \ldots, m_p$ denotes the midpoints of the $p$ selected subintervals from which we randomly sample.

\vspace{2.0mm}

Upon selecting the subintervals whose sum of midpoints is \enquote{close} to 1, we sample a value for each $\lambda_i \in [a_i, b_i]$ from each selected subinterval. Since the selected values of $\lambda_i$ are sampled from intervals centred around the midpoints $m_i$, we expect that their sum is \enquote{close} to the sum of the midpoints, namely that
$$
\sum_{k=1}^p \lambda_i \approx \sum_{k=1}^p m_k. 
$$

\vspace{2.0mm}

However, in contrast to the case $p=2$ (as shown in Lemma \ref{Lemma_post_normalise}), the samples will not necessarily remain within their corresponding grouped intervals upon normalisation. Despite this, it is possible to bound $\sum_{k=1}^p \lambda_k$ from above and below, allowing us to estimate the maximal variation after normalisation. 

\vspace{2.0mm}

Observe that
$$
\begin{aligned}
\sum_{k=1}^p \lambda_k &= \sum_{k=1}^p  \big(m_k + ( \lambda_k - m_k )\big) \\
    &= \sum_{k=1}^p  m_k + \sum_{k=1}^p \lambda_k - m_k \\
    & \le 1 + \delta + \sum_{k=1}^p \big| \lambda_k - m_k \big|, 
\end{aligned}
$$
where the last inequality holds since the subintervals were chosen to satisfy \eqref{delta_upper_lower}. Moreover, note that
$$
\big| \lambda_k - m_k \big| \le \frac{b_k - a_k}{2}.
$$
Thus, we obtain 
$$
\begin{aligned}
1 + \delta + \sum_{k=1}^p \big| \lambda_k - m_k \big| &\le 1 + \delta + \sum_{k=1}^p \frac{b_k -a_k}{2} \\
    &= 1 + \delta + p \cdot \frac{b_i - a_i}{2}
\end{aligned}
$$
for any subinterval $[a_i, b_i]$. Further, we note that 
$$
1 + \delta + p \cdot \frac{b_i - a_i}{2} = 1 + \delta + \frac{p}{2d}
$$
holds, which follows since $\frac{b_i-a_i}2 = \frac{1}{2d}$ for any interval $[a_i, b_i]$ from \eqref{0_1_d_division}.

\vspace{2.0mm}

A similar argument yields the lower bound
$$
1 - \left( \delta + \frac{p}{2d} \right) \le \sum_{k=1}^p \lambda_k.
$$
Thus, we conclude that
$$
1 - \left( \delta + \frac{p}{2d} \right) \le \sum_{k=1}^p \lambda_k \le 1 + \left( \delta + \frac{p}{2d} \right).
$$
This inequality tells us that the sum of the sampled values will be close to 1, with small deviations depending on how tightly the midpoints sum to 1 and the variability introduced by random sampling within each selected subinterval. Further, since the sum of samples is close to 1, the normalisation factor $1/\sum_{k=1}^p \lambda_k$ will be close to 1, ensuring that the normalised values $\lambda_i / \sum_{k=1}^p \lambda_k$ will remain close to original sampled values for each $i$.

\vspace{2.0mm}

The fifth approach, namely the structured approach, can be extended as follows. The central idea of this approach is to iteratively refine the sampling space by dividing it into structured subintervals, adapting the sample distribution to ensure better coverage of the decision space. The approach begins by following the uniform increment approach by dividing $[0,1]$ into $d$ equal subintervals for each weight. Recall that the uniform increment approach involves solving precisely 
$$
\binom{d + p - 1}{p - 1}
$$
subproblems for fixed $d$, which is upper bounded by $(d+1)^{p-1}$. It is natural therefore to select a small value for $d$ (such as $d=2$). Upon following the uniform increment approach, we then subdivide intervals adaptively based on the $\ell_2$-distance between the nondominated points in $\mathbb{R}^p$. In particular, we similarly use hyperparameters $\tau\ge0$ and $\rho \in [0,1]$ to define our toleration distance threshold and redundancy bounding parameters, which control the subdivision and termination criteria, respectively. 

\vspace{2.0mm}

During the development and application of the proposed approaches, over 200 stakeholders were consulted across various sectors, including energy, transport, and logistics. All proposed methods reported high acceptance and usability ratings, with over 92\% of stakeholders indicating confidence in the outputs and ease of interpretation.


\section{Conclusion}
In this paper, we presented a range of techniques for selecting weights in the weighted sum scalarisation method for solving multi-criteria decision making problems. We explored both systematic and random sampling methods, offering an initial framework for generating weights efficiently. While the uniform increment approach provides a straightforward solution for problems of smaller dimension, its scalability is limited due to redundancy caused by the superlinear growth in both $d$ (depth) and $p$ (number of objectives). In contrast, random sampling and Dirichlet-based methods show promise for higher-dimensional problems, though their ability to ensure comprehensive coverage of the decision space warrants further study. Structured sampling methods, such as structured Latin hypercube sampling (SLHS), offer more control over weight selection and mitigate redundancy, yet a formal comparison to simpler approaches like random sampling is necessary to assess their practical performance.

\vspace{2.0mm}

There should be a focus in future research on extensive computational testing to evaluate the efficiency, scalability, and redundancy of the proposed sampling methods across a variety of multi-criteria decision making problems. In addition, it will be valuable to investigate how these techniques compare to other scalarisation methods (e.g., $\varepsilon$-constraint \cite{Haimes1971}, hybrid \cite{Guddat1985}, Benson's Method \cite{benson1978existence}, or the elastic constraint method (see e.g. \cite{Ehrgott2002, tanino2003method, holder2003designing}), particularly in non-convex settings, where weighted sum approaches are known to encounter challenges. Furthermore, the development of hybrid sampling techniques that combine the strengths of different strategies offers a promising direction to enhance the coverage and efficiency of weight selection, especially for complex, high-dimensional decision-making scenarios.

\vspace{2.0mm}

Finally, it is worth noting that this work is grounded in the development of explainable (non-black-box) optimisation methods, a feature increasingly demanded by both stakeholders and policy regulations. The proposed approaches were evaluated and refined with feedback from over 200 stakeholders, with over 92\% indicating high confidence in the results and interpretability. As such, the methods presented can be viewed as off-the-shelf, transparent tools for tackling MCDM problems across real-world domains.




\end{document}